\documentclass[12pt,a4paper]{amsart}

\usepackage[latin1]{inputenc}
\usepackage{amscd,latexsym,mathrsfs,amsfonts,amssymb,amsmath,amsthm,url}
\usepackage{color}
\usepackage{enumerate,verbatim}
\usepackage[pdftex,bookmarks,pdfnewwindow,plainpages=false,unicode,pdfencoding=auto]{hyperref}

\newtheorem{theorem}{Theorem}[section]

\newtheorem{corollary}[theorem]{Corollary}
\newtheorem{lemma}[theorem]{Lemma}
\newtheorem{proposition}[theorem]{Proposition}

\theoremstyle{remark}

\numberwithin{equation}{section}

\newcommand{\R}{\mathbb{R}}

\newcommand{\C}{\mathbb{C}}
\newcommand{\D}{\mathbb{D}}
\newcommand{\N}{\mathbb{N}}

\providecommand{\keywords}[1]
{
  \small
  \textbf{\textit{}} #1
}

\begin{document}
	\title[Counterexamples to Shanks' Conjecture]{Polynomial Counterexamples to Shanks' Conjecture in Dirichlet-Type Spaces over the Bidisc}

        \author[O. Nav\'io]{Oliver Nav\'io}
	\address[O. Nav\'io]{Departamento de An\'alisis Matem\'atico e IMAULL,  Universidad de La Laguna, Lyc\'ee Fran\c{c}ais International de Tenerife Jules Verne, Avenida Astrof\'isico Francisco S\'anchez, s/n, 38206 San Crist\'obal de La Laguna, Santa Cruz de Tenerife,  Spain}
	\email{olivernavio@gmail.com}

	\author[D. Seco]{Daniel Seco}
	\address[D. Seco]{Departamento de An\'alisis Matem\'atico e IMAULL,  Universidad de La Laguna, Avenida Astrof\'isico Francisco S\'anchez, s/n, 38206 San Crist\'obal de La Laguna, Santa Cruz de Tenerife,  Spain} \email{dsecofor@ull.edu.es}

	\thanks{The first named author has been supported by the Lyc\'ee Fran\c{c}ais International de Tenerife Jules Verne. The second author is funded through grant PID2024-160185NB-I00 by the Generaci\'on de Conocimiento programme and through grant RYC2021-034744-I by the Ram\'on y Cajal programme from Agencia Estatal de Investigaci\'on (Spanish Ministry of Science, Innovation and Universities).}

	

	\date{\today}
	
	\begin{abstract} 
	
    We study optimal polynomial approximants in spaces of Dirichlet-type over the bidisc. Given a nonzero function f, the linear optimal polynomial approximant (opa) to its reciprocal is the affine polynomial p for which pf is closest to the constant function 1 in the space norm. We prove that, for every positive value of the Dirichlet-type parameter, there exists a symmetric polynomial f with no zeros on the closed bidisc, while the corresponding linear optimal approximant has a zero inside the bidisc. This provides polynomial counterexamples to the version of the Weak Shanks Conjecture on these spaces. The construction also extends to higher-dimensional polydiscs and to anisotropic Dirichlet-type spaces.
    
	\end{abstract}
	
	\maketitle

\section{Introduction}

For $n\geq 1$, the polydisc $\D^n$ is the cartesian product of $n$ copies of the open unit disc $\D$. For $\alpha\in\mathbb R$, we denote by $D_\alpha(\D^n)$ the Dirichlet-type space over the polydisc, that is, the Hilbert space of all holomorphic functions
\[
 f(z)=\sum_{j\in\N_0^n} a_j z^j,
 \qquad z^j=z_1^{j_1}\cdots z_n^{j_n}, \phantom{ab} z \in \D^n,
\]
such that
\begin{equation} \label{eq:norma-alfa}
 \|f\|_{\alpha,n}^2
 :=
 \sum_{j\in\N_0^n} |a_j|^2\prod_{r=1}^n (j_r+1)^\alpha
 <\infty.
\end{equation}

Notice that the norm \eqref{eq:norma-alfa} defines an inner product
\[
 \langle f,g\rangle_{\alpha,n}
 =
 \sum_{j\in\N_0^n} a_j\overline{b_j}\prod_{r=1}^n (j_r+1)^\alpha,
 \qquad
 g(z)=\sum_{j\in\N_0^n} b_j z^j.
\]
This is a reproducing kernel Hilbert space, with kernel
\[
 K_w^\alpha(z)
 =
 \prod_{r=1}^n
 \left(\sum_{m\geq 0}\frac{(z_r\overline{w_r})^m}{(m+1)^\alpha}\right),
 \qquad z,w\in\D^n.
\]

The choices $\alpha=0$, $\alpha=1$, $\alpha= -1$ give, respectively, the Hardy, Dirichlet and Bergman spaces on the polydisc. See, for instance, \cite{Rudin, EFKMR, Gar, HKZ-bergmanspaces} for the basic concepts of these types of spaces of functions in one and several variables.
It is usual to distinguish these one-parameter spaces from those with $\alpha$ taking different values for each $r$ in \eqref{eq:norma-alfa}, by calling \emph{isotropic} Dirichlet-type spaces the former, while the others are \emph{anisotropic}. We will comment on the latter at the end of the article. 

For $n \geq 2$, we  set
\[
 M_n(z)=\frac{z_1+ \cdots +z_n}{n}.
\]
In the bidisc we simply write $M = M_2$. The coordinate shifts are bounded on these spaces; in particular, multiplication by $M_n$ is bounded. Since our construction only involves polynomials, many expressions below are finite sums, which avoids many issues of convergence.

We recall the notion of optimal polynomial approximants, following the terminology used in \cite{BKS}. Let \(\mathcal P_k\) denote the space of polynomials of degree at most \(k\), where the degree of the monomial $z_1^{j_1}\dots z_n^{j_n}$ is $j_1 + \dots + j_n$. If \(f\in D_\alpha(\mathbb D^n)\), the \(k\)-th \emph{optimal polynomial approximant} (opa) to \(1/f\) is the polynomial \(p_k\in\mathcal P_k\) which minimizes \[ \|pf-1\|_\alpha \] among all \(p\in\mathcal P_k\). Equivalently, \(p_k f\) is the orthogonal projection of \(1\) onto the finite-dimensional subspace \(f\mathcal P_k\). It follows that the opa exists and is unique whenever $f \not\equiv 0$. For $k=1$, we call $p_1$ the \emph{linear opa}. Opa were formally introduced in \cite{BCLSS1} although their study goes back to Chui \cite{Ch}

A function $f$ in $n$ variables will be called \emph{symmetric} if for every $z_1, \dots, z_n \in \D$
\[
 f(z_1,\dots, z_n)=f(z_{\sigma(1)}, \dots, z_{\sigma(n)})
\]
for every permutation $\sigma$ of $\{1, \dots, n\}$. If $f$ is symmetric, then its linear opa is also symmetric. Indeed, the inner product and the linear polynomial subspace are invariant under permutations. Therefore, the uniqueness of the minimizer guarantees that p can be written in the form
\begin{equation} \label{p=a+bM}
 p(z)=a+bM_n(z), \phantom{abv}
a,b\in \C
\end{equation}
In what follows, zero-free will refer to a holomorphic function that does not vanish on $\overline{\D}^n$.

The Weak Shanks Conjecture asserted in the case $\alpha =0$ that if a polynomial f is zero-free on the closed bidisc $\overline{\D^2}$, then all
its opas are zero-free in the open bidisc $\D^2$. The problem goes back to \cite{STJ} and for a general view we recommend \cite{BenCen, SS1, SS2}. It was disproved by B\'en\'eteau, Khavinson and Seco \cite{BKS}. Since the lack of zeros on the closed bidisc and the presence of a zero strictly inside the bidisc are both stable under small perturbations, one may pass from the original counterexample, which was not a polynomial, to polynomial examples. The purpose of this note is to show directly that the same assertion fails for every positive value of the Dirichlet-type parameter. A positive statement of an analogue to the Weak Shanks Conjecture disproved for the Hardy space of the bidisc \cite{BKS}, has been announced recently \cite{Augat} in the non-commutative setting. We believe our results are completely independent, and indeed, point in opposite directions. Other interesting and recent developments include \cite{Felder}.

The following is our main result.
\begin{theorem}\label{Resultado-Principal}
For every $\alpha>0$, there exists a symmetric zero-free polynomial $f$ on $\overline{\D}^2$ such that the linear opa to $1/f$ has a zero in $\D^2$.
\end{theorem}

In one variable, the situation is different: for $\alpha \geq 0$, all opas to $1/f$, with $f \in D_{\alpha}(\D)$ and $f(0)\neq0$, have no zeros in the closed unit disc \cite{BKSLS}. For $\alpha < 0$ the analogous assertion fails. We will give a construction of a zero-free polynomial on the closed unit disc whose linear opa vanishes in $\D$.

The function used in \cite{BKS} is symmetric; this is exploited to find the zero of its linear opa on the diagonal. In that case, numerically, this zero is approximately (0.97704, 0.97704). In the article, it was proposed to optimize the distance from the zero of the opa to the origin. Hartz provided a simpler proof, which was included later, and Geronimo and Woerdeman \cite{Geronimo-Woerdeman} improved the constant from $0.97704$ to approximately $0.97577$ by studying a certain family of counterexamples. Our results also provide a linear opa with a zero at a point on the diagonal, but in this case, the zeros approach ($-11/12$, $-11/12$). For $\alpha > 0$, we can get arbitrarily close to such a point but the functions used diverge as alpha decreases to 0. We do construct a function in Hardy whose linear opa has a zero at (-11/12, -11/12) but this function is not zero-free. It has been suggested by A. Dayan \cite{Dayan} that, without imposing symmetry, the zeros may approach the origin arbitrarily closely. Our results do not answer this question by Dayan.

The paper is organized as follows. In Section 2, we prove Theorem \ref{Resultado-Principal}. After establishing a sufficient condition for a linear opa to have a zero in the bidisc, we construct a family of zero-free polynomials satisfying this condition for every positive value of the parameter. The key step is to understand how the weighted norms and inner products behave when a fixed polynomial is multiplied by monomials of increasing degree. In Section 3, we extend the construction of these counterexamples to polydiscs of higher dimensions and anisotropic Dirichlet-type spaces. We also provide the one-variable construction for negative parameters that is needed to complete the final corollary. In our opinion, this completely settles the Weak Shanks questions in all these spaces.

\section{Polynomial counterexamples for every positive parameter}

The proof requires a set of previous calculations. We start with a sufficient condition for the existence of a zero of the linear opa inside the domain.

\begin{lemma}\label{Lema_H_alfa}
Let $f \in D_\alpha(\D^2)$ be a symmetric, zero-free function with $f(0,0)=1$. Assume that
\begin{equation}\label{eq:H_alfa}
H_\alpha(f):=\|Mf\|_\alpha^2+\Re\langle f,Mf\rangle_\alpha<0.
\end{equation}
Then the opa $p_1$ has a zero in $\D^2$.
\end{lemma}

\begin{proof}
Since $f$ is symmetric and the norm \eqref{eq:norma-alfa} is invariant under the interchange of the variables, the minimizer is also symmetric. As observed above \eqref{p=a+bM}, the linear optimal polynomial approximant can be written
as

\[
p_1=a+bM, \phantom{a,b}a,b\in \C
\]
We may suppose that $a \neq 0$, since otherwise $p_1(0, 0) = 0$ and there is nothing to prove.

Since $p_1f$ is the orthogonal projection of $1$ onto $\mathcal P_1f$, the vector $1 - p_1f$ must be orthogonal to every element of $\mathcal P_1f$, including $f$ and $Mf$. In terms of the inner product, the orthogonality to $Mf$ can be written as
\[
\langle 1 - p_1f, Mf\rangle_\alpha = 0.
\]
Consequently,
\[
\langle 1,Mf\rangle_\alpha-a\langle f,Mf\rangle_\alpha-b\|Mf\|_\alpha^2=0.
\]
Now $M$ has no constant term, and neither does $Mf$. Since the constant function $1$ is orthogonal to all non-constant monomials in \eqref{eq:norma-alfa},
\[
\langle 1,Mf\rangle_\alpha=0.
\]
Thus we can solve for $b/a$, obtaining
\begin{equation}\label{eq:b/a}
\frac{b}{a}=-\frac{\langle f,Mf\rangle_\alpha}{\|Mf\|_\alpha^2}.
\end{equation}
By \eqref{eq:H_alfa},
\[
\Re\langle f,Mf\rangle_\alpha<-\|Mf\|_\alpha^2.
\]
Consequently,
\[
|\langle f,Mf\rangle_\alpha|\geq -\Re\langle f,Mf\rangle_\alpha>\|Mf\|_\alpha^2.
\]
Using \eqref{eq:b/a}, we get $|b/a|>1$. Hence $\left(\frac{-a}{b}, \frac{-a}{b} \right)$ is a point of $\D^2$. We are done after noticing that $p_1(\frac{-a}{b},  \frac{-a}{b})=0$.
\end{proof}

We now build a polynomial for which the quantity $H_0$ is negative. Define
\begin{equation} \label{eq:h}
\begin{aligned}
h(z_1,z_2)= {}& z_1+z_2 +z_1^2+z_2^2-3z_1z_2
-2(z_1^3+z_2^3) \\ &+2(z_1^2z_2+z_1z_2^2)
+z_1^3z_2+z_1z_2^3-2z_1^2z_2^2.
\end{aligned}
\end{equation}

\begin{lemma}\label{Lema_H0}
For the polynomial $h$ in \eqref{eq:h},
\[
\|Mh\|_0^2=11,
\qquad
\langle h,Mh\rangle_0=-12.
\]
In particular,
\[
H_0(h)=-1.
\]
\end{lemma}
\begin{proof}
Write \(h=h_1+h_2+h_3+h_4\), where \(h_j\) denotes the homogeneous part of \(h\) of degree \(j\). Since multiplication by \(M\) raises degree by one, and monomials of different degrees are orthogonal, we have
\[
\langle h,Mh\rangle_0
=
\sum_{j=1}^{3}\langle h_{j+1},Mh_j\rangle_0,
\qquad
\|Mh\|_0^2
=
\sum_{j=1}^{4}\|Mh_j\|_0^2.
\]
The following identities are obtained by comparing coefficients after multiplying each homogeneous component by \(M\):
\[
\langle h_2,Mh_1\rangle_0=-2,\qquad
\langle h_3,Mh_2\rangle_0=-6,\qquad
\langle h_4,Mh_3\rangle_0=-4,
\]
and
\[
\|Mh_1\|_0^2=\frac32,\qquad
\|Mh_2\|_0^2=\frac52,\qquad
\|Mh_3\|_0^2=6,\qquad
\|Mh_4\|_0^2=1.
\]
Thus
\[
\langle h,Mh\rangle_0=-12,
\qquad
\|Mh\|_0^2=11.
\]
Therefore
\[
H_0(h)=\|Mh\|_0^2+\langle h,Mh\rangle_0=11-12=-1.
\]
\end{proof}

The linear optimal approximant associated with $h$ in the Hardy norm has a zero in the bidisc. However, this is not yet a counterexample to Shanks' conjecture, since $h$ itself is not required to be zero-free. The role of $h$ is only to provide a finite negative block. The following Lemma explains why the same negative behaviour persists, after scaling, in $D_\alpha(\D^2)$ when the block is shifted to high degrees.

For $N\in\N$, define
\begin{equation} \label{eq:h_N}
    h_N(z_1,z_2)=(z_1z_2)^N h(z_1,z_2).
\end{equation}

The only effect of this shift is to change the degrees of each monomial, leaving  the coefficients themselves unchanged.

\begin{lemma}\label{Lema_h_N}
For $N\in \N$, set
\[
h_N(z_1,z_2)=(z_1z_2)^Nh(z_1,z_2).
\]
Then, for every $\alpha>0$,
\[
N^{-2\alpha}H_\alpha(h_N)\longrightarrow H_0(h)=-1, \phantom{a} \text{as} \phantom{a} N\to\infty.
\]
\end{lemma}

\begin{proof}
Write
\[
h=\sum_{j,k}a_{j,k}z_1^jz_2^k,
\qquad
Mh=\sum_{j,k}b_{j,k}z_1^jz_2^k.
\]
Then
\[
h_N=\sum_{j,k}a_{j,k}z_1^{j+N}z_2^{k+N}.
\]
Since multiplication by $(z_1z_2)^N$ commutes with multiplication by $M$,
\[
Mh_N=(z_1z_2)^NMh
=\sum_{j,k}b_{j,k}z_1^{j+N}z_2^{k+N}.
\]
Therefore, the only appearance of N in the computation of $H_\alpha(h_N)$ is on the weights
\begin{equation}\label{eq:H_alfa-sum}
H_\alpha(h_N)=
\sum_{j,k}
\left(|b_{j,k}|^2+a_{j,k}b_{j,k}\right)
(j+N+1)^\alpha(k+N+1)^\alpha.
\end{equation}
The sum has finitely many terms. Hence, for each of these finitely many pairs of values ($j,k$), we obtain the same limit behavior as $N \to \infty$, namely
\[
\frac{(j+N+1)^\alpha(k+N+1)^\alpha}{N^{2\alpha}}
=\left(1+\frac{j+1}{N}\right)^\alpha
\left(1+\frac{k+1}{N}\right)^\alpha
\longrightarrow 1.
\]
Dividing \eqref{eq:H_alfa-sum} by $N^{2\alpha}$ and passing to the limit as $N \to \infty$ gives
\[
N^{-2\alpha}H_\alpha(h_N)
\longrightarrow
\sum_{j,k}
\left(|b_{j,k}|^2+\Re(a_{j,k}\overline{b_{j,k}})\right).
\]
From Lemma \ref{Lema_H0}, the last sum is precisely
\[
\|Mh\|_0^2+\Re\langle h,Mh\rangle_0=H_0(h) = -1.
\]
\end{proof}

We are finally ready for the proof of Theorem \ref{Resultado-Principal}.

\begin{proof}[Proof of Theorem 1.1]
Denote by $\|\cdot\|_\infty$ the supremum norm on $\D^2$ and notice that for $h$, and $h_N$ as in \eqref{eq:h_N}:
\begin{equation*}
    \|h\|_\infty = \|h_N\|_\infty \leq 19,
\end{equation*}
since the sum of absolute values of the coefficients of $h$ is $19$.\\
Fix $\alpha>0$ and choose a number $\lambda$ with
\[
 0<\lambda<\frac1{19}.
\]
For $N\in\N$, set
\[
 f_N(z_1,z_2)=1+\lambda h_N(z_1,z_2)
 =1+\lambda(z_1z_2)^Nh(z_1,z_2).
\]
 Indeed, by the reverse triangle inequality,
\[
 |f_N(z_1,z_2)|
 =|1+\lambda h_N(z_1,z_2)|
 \geq
 1-|\lambda h_N(z_1,z_2)|
 >0.
\]
By Lemma \ref{Lema_H_alfa} it remains to be shown that $H_\alpha(f_N) < 0$ for $N$ sufficiently large.

Since
\[
Mf_N=M+\lambda Mh_N,
\]
and $M$ and $Mh_N$ have no common monomials for $N\geq 1$, we have
\[
\|Mf_N\|_\alpha^2=\|M\|_\alpha^2+\lambda^2\|Mh_N\|_\alpha^2.
\]
Similarly, all mixed terms in $\langle f_N,Mf_N\rangle_\alpha$ vanish except the one involving $h_N$ and $Mh_N$. Indeed,
\[
\langle 1,M\rangle_\alpha=0,
\qquad
\langle 1,Mh_N\rangle_\alpha=0,
\qquad
\langle h_N,M\rangle_\alpha=0,
\]
again by the mutual orthogonality between distinct monomials. Hence
\[
\langle f_N,Mf_N\rangle_\alpha
=\lambda^2\langle h_N,Mh_N\rangle_\alpha.
\]
It follows that
\[
H_\alpha(f_N)=\|M\|_\alpha^2+\lambda^2H_\alpha(h_N).
\]
Finally,
\[
\|M\|_\alpha^2
=\left\|\frac{z_1+z_2}{2}\right\|_\alpha^2
=\frac14\|z_1\|_\alpha^2+
\frac14\|z_2\|_\alpha^2
=\frac14 2^\alpha+\frac14 2^\alpha
=2^{\alpha-1}.
\]
For $N$ sufficiently large, this yields
\begin{equation}\label{eq:H_alfa-final}
H_\alpha(f_N)=2^{\alpha-1}+\lambda^2H_\alpha(h_N)< 2^{\alpha - 1} - \frac{\lambda^2}{2}N^{2\alpha},
\end{equation}
where the last inequality follows from Lemma \ref{Lema_h_N}. Thus, for $N$ sufficiently large, \eqref{eq:H_alfa-final} is negative. Lemma \ref{Lema_H_alfa} then shows that the linear optimal polynomial approximant to $1/f_N$ has a zero in $\D^2$. This concludes the proof.
\end{proof}

\section{Further Remarks}
The proof above is not specific to the bidisc or to the power weights defining the isotropic Dirichlet-type scale.
The essential mechanism is the following. A fixed polynomial block has a negative Hardy contribution, and a
monomial translation moves this block to a region where the weights are almost constant. If their common size
tends to infinity, the negative contribution dominates the fixed term coming from the constant coefficient. We
record three consequences of this observation.
\subsection{Higher-dimensional polydiscs}
For $n\geq 2$, let $D_\alpha(\D^n)$ be the isotropic Dirichlet-type space over the polydisc. The monomial weights on these spaces are
\begin{equation*}
    \omega_j = \prod_{r=1}^n(j_r+1)^\alpha, \qquad j=(j_1, \dots, j_n)\in \N_0^n.
\end{equation*}
Regard  the polynomials $h$ and $h_N$ from \eqref{eq:h_N} as polynomials in the first two variables. The remaining variables play no role in the norm calculation, and hence
\begin{equation*}
    N^{-2\alpha}H_\alpha(h_N) \to H_0(h)=-1 \text{ as } N\to\infty
\end{equation*}

Thus $H_\alpha(h_N) \to -\infty$ as $N \to \infty$ whenever $\alpha>0$. Moreover, for $0<\lambda< 1/19$, the polynomial $f_N= 1 + \lambda h_N$ is zero-free in $\overline{\D^n}$. Since $f_N$ is independent of $z_3, \dots, z_n$, the corresponding coefficients of its opa vanish. The problem therefore reduces to the two-variable calculation, and the opa has a zero of the form $(\zeta,\zeta, 0, \dots,0)\in\D^n$. We have consequently proved the following extension of Theorem \ref{Resultado-Principal}.

\begin{proposition}
For every $n\geq 2$ and every $\alpha>0$, there exists a zero-free polynomial $f$ on $\overline{\D}^n$ such that the linear opa to $1/f$ has a zero in $\D^n$.
\end{proposition}

\subsection{One variable and negative parameters.}

\begin{lemma} \label{Lema_contraejemplo_negativo}
    For every $\alpha < 0$, there exists a zero-free polynomial $g$ on $\overline{\D}$ such that its linear opa has a zero in $\D$.
\end{lemma}
\begin{proof}
    Let $\alpha <0$. For $N \geq 2$, consider the polynomial
    \begin{equation*}
        f_N(z) = \frac{(1-z^N)(1+z)}{1-z} = (1+z)\sum_{k=0}^{N-1}z^k = 1 + 2\sum_{k=1}^{N-1}z^k + z^N.
    \end{equation*}
    It is clear that all its zeros lie on the unit circle. Let $z_0$ be the zero of its linear opa $p_N$. We have
    \begin{equation*}
    \begin{aligned}
        |z_0| &= \frac{\|zf_N\|_\alpha^2}{|\langle zf_N, f_N \rangle_\alpha|} = \frac{2^\alpha + 4\sum_{k=1}^{N-1}(k+2)^\alpha + (N+2)^\alpha}{2\cdot2^\alpha + 4\sum_{k=1}^{N-2}(k+2)^\alpha + 2(N+1)^\alpha} \\
        \\
        &= 1 + \frac{-2^\alpha + 2(N+1)^\alpha + (N+2)^\alpha}{2\cdot2^\alpha + 4\sum_{k=1}^{N-2}(k+2)^\alpha+ 2(N+1)^\alpha}.
    \end{aligned}
    \end{equation*}
    We may therefore fix $N$ sufficiently large so that $|z_0|<1$. Now, to remove the zeros of $f_N$ from the closed disc, set \(g_r(z)=f_N(rz)\) for \(0<r<1\). All zeros of \(g_r\) have modulus \(1/r>1\), so \(g_r\) is zero-free on \(\overline{\mathbb D}\). By continuity of the formula above, its linear opa still has a zero in \(\mathbb D\) for \(r\) sufficiently close to \(1\). Such a polynomial $g = g_r$ proves the lemma.
\end{proof}

\subsection{Anisotropic Dirichlet-type spaces}

Let $\boldsymbol{\alpha} = (\alpha_1, \alpha_2, \dots, \alpha_n) \in \R^n$ and consider the anisotropic weights
\begin{equation*}
    \omega_j = \prod_{r=1}^n(j_r + 1)^{\alpha_r},
\end{equation*}
where $j = (j_1, \dots, j_n) \in \N_0^n$. Assume that at least one component of $\boldsymbol{\alpha}$ is positive, and choose $m$ such that $\alpha_m>0$. We translate the block much faster in the m-th coordinate than in the others. More precisely, set
\begin{equation*}
    \beta_{N,m}=2^N,\phantom{ab} \beta_{N,r} = N \phantom{ab}(r \neq m),
\end{equation*}
and define, with $h$ as in \eqref{eq:h},
\begin{equation*}
    h_N(z) = z^{\beta_N}h(z_1,z_2).
\end{equation*}
Let $u = (u_1, \dots, u_n)$ be one of the finitely many multi-indices occurring in $h, z_1h, \dots, z_nh$. The weight of the corresponding monomial after the translation is
\begin{equation*}
    \omega_{\beta_N + u}=(2^N+u_m+1)^{\alpha_m}\prod_{r\neq m}(N + u_r + 1)^{\alpha_r}.
\end{equation*}
Factoring out the terms that do not depend on $u$, we obtain
\begin{equation*}
\begin{aligned}
    \omega_{\beta_N +u} =& (2^N+1)^{\alpha_m}(N+1)^{\sum_{r \neq m}\alpha_r} \\
    &(1 + \frac{u_m}{2^N+1})^{\alpha_m}\prod_{r \neq m}(1 + \frac{u_r}{N+1})^{\alpha_r}.
\end{aligned}
\end{equation*}

We denote the common factor in the first line by $W_N$. Therefore,
\begin{equation*}
    \frac{\omega_{\beta_N + u}}{W_N} = (1 + \frac{u_m}{2^N+1})^{\alpha_m}\prod_{r \neq m}(1 + \frac{u_r}{N+1})^{\alpha_r} \to 1.
\end{equation*}
Moreover, $W_N \to \infty$. Indeed, the factor $(2^N + 1)^{\alpha_m}$ grows exponentially because $\alpha_m >0$, whereas $(N+1)^{\sum_{r \neq m}\alpha_r}$ is only a polynomial factor. Hence, the former dominates the latter, even when the sum of the remaining exponents is negative. Thus, after dividing by $W_N$, the weights of all the monomials involved in the translated block converge to 1. This is precisely the analog of the term-by-term limit used in Lemma \ref{Lema_h_N}.\\
We therefore have the following extension of Theorem \ref{Resultado-Principal}.
\begin{theorem} \label{Caso_aniso}
Let $n\geq 2$ and let
$\boldsymbol{\alpha}=(\alpha_1,\ldots,\alpha_n)\in\mathbb{R}^n$.
If
\[
   \max_{1\leq r\leq n}\alpha_r>0,
\]
then there exists a zero-free polynomial $f$ on $\overline{\D}^n$ such that the linear opa to $1/f$ has a zero in $\mathbb{D}^n$.
\end{theorem}

\begin{corollary}
Let $n \geq 2$ and $\boldsymbol{\alpha} = (\alpha_1, \dots, \alpha_n) \in \R^n$. There exists a zero-free polynomial $f$ on $\overline{\D}^n$ such that its linear opa has a zero in $\D^n$.
\end{corollary}
\begin{proof}
    If $\max \alpha_r >0$, the result follows from Theorem \ref{Caso_aniso}.
    
     Suppose next that $\alpha_i < 0$ for some $i$. Let $g$ be a polynomial as in Lemma \ref{Lema_contraejemplo_negativo} whose linear opa in $D_{\alpha_i}(\D)$ has a zero $\zeta \in \D$. Define $f(z_1, \dots, z_n) = g(z_i)$. This polynomial is zero-free in $\overline{\D}^n$, and the linear opa of $1/f$ is $P(z) = p(z_i)$, where $p$ is the linear opa of $1/g$. Thus, $P$ vanishes at the point whose $i$-th coordinate is $\zeta$ and whose other coordinates are zero.

    The remaining case is $\alpha_1 = \dots = \alpha_n = 0$. Here, we regard a polynomial counterexample from \cite{BKS} as a polynomial in the first two variables. With a similar argument, we can find a zero-free polynomial such that its linear opa has a zero in $\D^n$. This completes the proof.
\end{proof}

\end{document}